\documentclass[11pt, english]{article}
\usepackage[margin=3cm]{geometry}
\usepackage{amsbsy,amsmath,amsthm,amssymb}

\usepackage{fourier}
\DeclareMathAlphabet{\mathcal}{OMS}{cmsy}{m}{n}

\usepackage{xcolor,bm}
\usepackage{esint}
\usepackage{babel}
\usepackage[colorlinks=true, citecolor = blue]{hyperref}

\usepackage{indentfirst}
\usepackage{changepage}	
\usepackage{setspace}
\usepackage{indentfirst}
\usepackage{graphicx}
\usepackage{calc}

\allowdisplaybreaks

\theoremstyle{plain}
\newtheorem{theorem}{Theorem}[section]
\newtheorem{lemma}[theorem]{Lemma}

\theoremstyle{remark}
\newtheorem{remark}[theorem]{Remark}

\theoremstyle{definition}
\newtheorem{definition}{Definition}

\newenvironment{proof of theorem 1.1}{{\noindent \em Proof of Theorem 1.1.}}{\hfill $\Box$\par}
\newenvironment{proof of theorem 1.2}{{\noindent \em Proof of Theorem 1.2.}}{\hfill $\Box$\par}

\DeclareSymbolFont{EulerExtension}{U}{euex}{m}{n}
\DeclareMathSymbol{\euintop}{\mathop} {EulerExtension}{"52}
\DeclareMathSymbol{\euointop}{\mathop} {EulerExtension}{"48}

\begin{document}
	\title{A note on the Laplace transforms of certain generalized fractional integral operators}
	\author{Min-Jie Luo$^{\rm 1}$\thanks{Corresponding author}, Jing-Yi Shen$^{\rm 2}$, Ravinder Krishna Raina$^{\rm 3}$}
	\date{}
	\maketitle
	\begin{center}\small
		$^{1}$\emph{Department of Mathematics, School of Mathematics and Statistics, \\ Donghua University, Shanghai 201620, \\ 
			People's Republic of China.}\\
		E-mail: \texttt{mathwinnie@live.com}, \texttt{mathwinnie@dhu.edu.cn}
	\end{center}
	\begin{center}\small 
		$^{2}$\emph{Department of Mathematics, School of Mathematics and Statistics, \\ Donghua University, Shanghai 201620, \\ 
			People's Republic of China.}\\
		E-mail: \texttt{Shen031127@outlook.com}
	\end{center}
	\begin{center}\small
		$^{3}${\emph{M.P. University of Agriculture and Technology, Udaipur (Rajasthan), India\\
				\emph{Present address:} 10/11, Ganpati Vihar, Opposite Sector 5,\\
				Udaipur-313002, Rajasthan, India.}\\
			E-mail: \texttt{rkraina1944@gmail.com}; \texttt{rkraina\_7@hotmail.com}}  
	\end{center}
	
	%	\vspace{0.5cm}
	
	\begin{abstract}
		
		In this paper, we derive certain formulas giving the Laplace transforms of two generalized fractional integral operators introduced recently in [Fract. Calc. Appl. Anal. 20 (2) (2017), 422--446]. The main results provide generalizations to various known results. Some useful remarks related to the results presented in this paper are also mentioned. \\
		
		\noindent\textbf{Keywords}: 
		Fox-Wright function, 
		Fractional integral operator, 
		Laplace transform, 
		$H$-function.
		\\
		
		\noindent\textbf{Mathematics Subject Classification (2020)}:
		26A33; % Fractional derivatives and integrals
		33C60. % Hypergeometric integrals and functions defined by them (E, G, H and I functions)
	\end{abstract}

	\section{Introduction}\label{Introduction}

The study of the Laplace transforms of various types of fractional integral and derivative operators is a basic and fundamental area of study in Fractional Calculus which has extensive applications in solving fractional integral and differential equations (see \cite{Gorenflo-Mainardi-1997}, \cite[Chapter 5]{Kilbas-Srivastava-Trujillo-Book-2006}, \cite{Luo-Raina-2025} and \cite[Chapter 4]{Podlubny-1999}). Since 2017, the first and third authors of this paper have published a series of papers (\cite{Luo-Raina-2017, Luo-Raina-2019, Luo-Raina-2022}) introducing a pair of fractional integral operators whose kernels involve a very special class of generalized hypergeometric function ${}_{r+2}F_{r+1}$ (see Definition \ref{New Definition} below). The Laplace transforms of these operators have so far not been investigated in detail, as suggested in \cite[pp. 79--80]{Luo-Raina-2025}. The purpose of this paper is therefore to study the Laplace transformation of such generalized forms of fractional calculus operators.

For any $m\in\mathbb{Z}$, let $\mathbb{Z}_{\geq m}:=\{n\in\mathbb{Z};n\geq m\}$, and also let $\mathbb{R}_{>0}:=\{x\in\mathbb{R};x>0\}=(0,\infty)$. As usual, the Pochhammer  symbol $(a)_k$ is defined by
\[
\left(a\right)_{k}:=\frac{\Gamma\left(a+k\right)}{\Gamma\left(a\right)}=
\begin{cases}
	1 & \left(k=0; ~ a\in\mathbb{C}\setminus\{0\}\right),\\
	a\left(a+1\right)\cdots\left(a+k-1\right) & \left(k\in\mathbb{N}; ~ a\in\mathbb{C}\right).
\end{cases}
\]
We shall use the convention of writing the finite sequence parameters $a_1$, $\cdots$, $a_p$ by $(a_p)$ and the product of $p$ Prochhammer symbols by $((a_p))_k\equiv (a_1)_k\cdots (a_p)_k$, where an empty product $p=0$ is treated as unity. The generalized hypergeometric function ${}_{p}F_{q}$ is then defined by the series
\begin{equation*}
	{}_{p}F_{q}\left[\begin{matrix}
		(a_p)\\
		(b_q)\end{matrix};z
	\right]
	:=\sum_{k=0}^{\infty}\frac{((a_p))_k}{((b_p))_k}
	\frac{z^k}{k!}~(|z|<1).
\end{equation*}
For its conditions of convergence and analytic continuation via Mellin-Barnes type integral, we refer the reader to \cite[pp. 30--31]{Kilbas-Srivastava-Trujillo-Book-2006} and \cite{NIST Handbook}.

\begin{definition}[{\cite[p. 423]{Luo-Raina-2017}}]\label{New Definition}
	Let $x, h, \nu\in\mathbb{R}_{>0}$, $\delta, a, b, f_1, \cdots, f_r\in\mathbb{C}$ and $m_1,\cdots,m_r\in\mathbb{Z}_{\geq0}$. Also, let $\Re(\mu)>0$ and $\varphi(s)$ be a suitable complex-valued function defined on $\mathbb{R}_{>0}$. Then, the fractional integral of a function $\varphi(x)$ of the first kind is defined by
	\begin{align}\label{Def New operator I}
		\left(\mathcal{I}\varphi\right)(x)
		&\equiv\left(\mathcal{I}\begin{smallmatrix}
			\mu; a,b: &(f_r+m_r)\\
			h; \nu,\delta: &(f_r)
		\end{smallmatrix}\varphi\right)(x)\notag\\
		&:=\frac{{\nu}x^{-\delta-\nu(\mu+ h)}}{\Gamma(\mu)}
		\int_{0}^{x}(x^\nu-s^\nu)^{\mu-1}
		{}_{r+2}F_{r+1}\left[\begin{matrix}
			a,b,\\
			\mu,
		\end{matrix}~
		\begin{matrix}
			(f_r+m_r)\\
			(f_r)
		\end{matrix};1-\frac{s^\nu}{x^\nu}\right]
		\varphi(s)s^{\nu h+\nu-1}\mathrm{d}s,
	\end{align}
	and the fractional integral of the second kind of a function $\varphi\left(x\right)$ is defined by
	\begin{align}\label{Def New operator II}
		\left(\mathcal{J}\varphi\right)(x)
		&\equiv\left(\mathcal{J}\begin{smallmatrix}
			\mu; a,b: &(f_r+m_r)\\
			h; \nu,\delta: &(f_r)
		\end{smallmatrix}\varphi\right)(x)\notag\\
		&:=\frac{\nu x^{\nu h+\nu-1}}{\Gamma(\mu)}
		\int_{x}^{\infty}(s^\nu-x^\nu)^{\mu-1}
		{}_{r+2}F_{r+1}\left[\begin{matrix}
			a,b,\\
			\mu,
		\end{matrix}~
		\begin{matrix}
			(f_r+m_r)\\
			(f_r)
		\end{matrix};1-\frac{x^\nu}{s^\nu}\right]
		\varphi(s)s^{-\delta-\nu(\mu+ h)}\mathrm{d}s.
	\end{align}
\end{definition}

Operators \eqref{Def New operator I} and \eqref{Def New operator II} include many important known fractional integral operators as special cases, such as the Riemann-Liouville operators, the Erd\'{e}lyi-Kober operators and the Saigo operators (see \cite{Luo-Raina-2022}). 

In Section \ref{Preliminaries}, we shall present some useful results regarding \eqref{Def New operator I} and \eqref{Def New operator II} and provide definitions of some special functions to be used later. In Section \ref{Main results}, we prove our main theorems, namely, the formulas for
\[
\mathcal{L}[x^\lambda \mathcal{I}\varphi](s)
~~~\text{and}~~~
\mathcal{L}[x^\lambda \mathcal{J}\varphi](s),
\]
where $\lambda\in\mathbb{C}$ and $\mathcal{L}$ denotes the Laplace transform of the function $\varphi(t)$ defined by 
\[
\mathcal{L}[\varphi](s):=\int_{0}^{\infty}\mathrm{e}^{-st}\varphi(t)\mathrm{d}t.
\]

\section{Preliminaries}\label{Preliminaries}

Following the earlier work in \cite{Luo-Raina-2017, Luo-Raina-2019} and \cite{Luo-Raina-2022}, we use for convenience sake the notations $\mathfrak{c}_1(t)$ and $\mathfrak{c}_2(t)$ given by
\begin{equation}\label{c1-c2}
	\left.\begin{aligned}
		\mathfrak{c}_1(t)&:=1+h+\frac{t}{\nu},\\
		\mathfrak{c}_2(t)&:=\mathfrak{c}_1(\delta-1)-\frac{t}{\nu}.
	\end{aligned}~~\right\}
\end{equation}
Also, for $k\in\mathbb{Z}_{\geq0}$, we define
\begin{equation}\label{pk}
	\mathfrak{p}_k:=\mu-a-b-k.
\end{equation}

Throughout the present paper, the sequence $A_k$ $\left(0\leq k\leq m:=m_1+\cdots+m_r\right)$ is always defined by
\begin{equation}\label{Miller summation formula II-Def of Ak}
	A_k=\sum_{j=k}^{m} {j\brace k}\sigma_{m-j}, ~~
	A_0=(f_1)_{m_1}\cdots(f_r)_{m_r}, ~~ A_m=1,
\end{equation}
where ${j\brace k}$ denotes the Stirling numbers of the second kind and $\sigma_j$ $(0\leq j\leq m)$ are generated by the relation
\[
(f_1+x)_{m_1}\cdots(f_r+x)_{m_r}=\sum_{j=0}^{m}\sigma_{m-j}x^j.
\]

\begin{lemma}[{\cite[p. 426]{Luo-Raina-2017}}]\label{Lemma FI of Power Function}
	Let $x, h, \nu\in\mathbb{R}_{>0}$, $\delta, \lambda, a, b, f_1,\cdots,f_r\in\mathbb{C}$ and $m_1,\cdots,m_r\in\mathbb{Z}_{\geq0}$.
	Also, let $\mathfrak{c}_1(t)$, $\mathfrak{c}_2(t)$ and
	$\mathfrak{p}_k$ be defined, respectively, by \eqref{c1-c2} and \eqref{pk}.
	Then there holds the following formulas:
	\begin{equation*} %\label{Lemma FI of Power Function I}
		\mathcal{I} x^{\lambda}
		=x^{\lambda-\delta}\sum_{k=0}^{m}\frac{A_k}{A_0}
		\frac{(a)_k(b)_k
			\Gamma(\mathfrak{c}_1(\lambda))
			\Gamma(\mathfrak{c}_1(\lambda)+\mathfrak{p}_k)}{
			\Gamma(\mathfrak{c}_1(\lambda)+\mu-a)
			\Gamma(\mathfrak{c}_1(\lambda)+\mu-b)},
	\end{equation*}
	provided that $\Re(\mu)>0$ and
	$\Re(\mathfrak{c}_1(\lambda))
	>-\min[0,\Re(\mathfrak{p}_m)]$, 
	and
	\begin{equation}\label{Lemma FI of Power Function II}
		\mathcal{J} x^{\lambda}
		=x^{\lambda-\delta}\sum_{k=0}^{m}\frac{A_k}{A_0}
		\frac{(a)_k(b)_k
			\Gamma(\mathfrak{c}_2(\lambda))
			\Gamma(\mathfrak{c}_2(\lambda)+\mathfrak{p}_k)}{
			\Gamma(\mathfrak{c}_2(\lambda)+\mu-a)
			\Gamma(\mathfrak{c}_2(\lambda)+\mu-b)},
	\end{equation}
	if $\Re(\mu)>0$ and
	$
	\Re(\mathfrak{c}_2(\lambda))
	>-\min[0,\Re(\mathfrak{p}_m)],
	$
	where $A_k$ $(0\leq k\leq m)$ is defined by \eqref{Miller summation formula II-Def of Ak}.
	\end{lemma}
	
Let $m, n, p, q \in \mathbb{Z}_{\geq0}$, $0 \leq m \leq q$, $0 \leq n \leq p$ and let $a_{i}, b_j\in \mathbb{C}, A_i, B_j\geq0$, where $1\leq i \leq p, 1\leq j \leq q$. The Fox $H$-function is defined by (see \cite[p. 1]{Kilbas-Saigo-Book-2004} and \cite[p. 58]{Kilbas-Srivastava-Trujillo-Book-2006}):
\begin{align}\label{H-Function-Def}
	H_{p, q}^{m, n}\left[z \left|\begin{matrix}
		(a_i, A_i)_{1,p} \\
		(b_j, B_j)_{1,q}
	\end{matrix}\right.\right]
	&=H_{p, q}^{m, n}\left[z \left|\begin{matrix}
		(a_1, A_1), \ldots,(a_p, A_p) \\
		(b_1, B_1), \cdots,(b_q, B_q)
	\end{matrix}\right.\right] \notag\\
	&=\frac{1}{2 \pi \mathrm{i}} \int_L \frac{\displaystyle\prod_{k=1}^m \Gamma(B_k s+b_k) \prod_{i=1}^n \Gamma(1-a_i-A_i s)}{\displaystyle \prod_{k=m+1}^q \Gamma(1-b_k-B_k s) \prod_{i=n+1}^p \Gamma(A_i s+a_i)}  z^{-s} \mathrm{d}s,
\end{align}
where $L$ is a suitable contour that separates the poles of  $\Gamma\left(B_k s+b_k\right)$ from the poles of  $\Gamma\left(1-a_i-A_i s\right)$. To further clarify the definition, we take $L=L_{\mathrm{i}c\infty}$, which is a contour starting at the point $c-\mathrm{i}\infty$ and terminating at the point $c+\mathrm{i}\infty$, where $c\in\mathbb{R}$. The properties of the $H$-function depend on the following indexes:
\begin{align}
	a^{*}&:=\sum_{i=1}^{n}A_i-\sum_{i=n+1}^{p}A_i
	+\sum_{j=1}^{m}B_j-\sum_{j=m+1}^{q}B_j,
	\label{H-function-Condition-1}\\
	\Delta&:=\sum_{j=1}^q B_j-\sum_{i=1}^p A_i,
	\label{H-function-Condition-2}\\
	\delta^{*}&:=\prod_{j=1}^{p}A_j^{-A_j}\prod_{j=1}^{q}B_j^{B_j},
	\label{H-function-Condition-3}\\
	\mu^{*}&:=\sum_{j=1}^{q}b_j-\sum_{i=1}^{p}a_i+\frac{p-q}{2},
	\label{H-function-Condition-4}\\
	a_1^{*}&=\sum_{j=1}^m B_j-\sum_{i=n+1}^p A_i.
	\label{H-function-Condition-5}
\end{align}
When $A_i=B_j=1$ $(i=1,\cdots,p; j=1,\cdots,q)$, the $H$-function \eqref{H-Function-Def} reduces to Meijer's $G$-function 
\begin{align*}
	G_{p, q}^{m, n}\left[z \left|\begin{matrix}
		(a_i, 1)_{1,p} \\
		(b_j, 1)_{1,q}
	\end{matrix}\right.\right]
	&=G_{p, q}^{m, n}\left[z \left|\begin{matrix}
		a_1, \cdots, a_p \\
		b_1, \cdots, b_q
	\end{matrix}\right.\right] \notag\\
	&=\frac{1}{2 \pi \mathrm{i}} \int_L \frac{\displaystyle\prod_{k=1}^m \Gamma(s+b_k) \prod_{i=1}^n \Gamma(1-a_i-s)}{\displaystyle \prod_{k=m+1}^q \Gamma(1-b_k-s) \prod_{i=n+1}^p \Gamma(s+a_i)}  z^{-s} \mathrm{d}s,
\end{align*}
where $L$ is the same contour taken for the $H$-function defined in \eqref{H-Function-Def}. We also have \cite[p. 67, Eq. (1.12.68)]{Kilbas-Srivastava-Trujillo-Book-2006}:
\begin{equation}\label{H-function-to-Fox-WrightF}
	H_{p, q+1}^{1, p}\left[z \left|\begin{matrix}
		(1-a_i, \alpha_i)_{1,p} \\
		(0,1), (1-b_j, \beta_j)_{1,q}
	\end{matrix}\right.\right]
	={}_{p}\Psi_{q}\left[\left.\begin{matrix}
		(a_j,\alpha_j)_{1,p}\\
		(b_j,\beta_j)_{1,q}
	\end{matrix}\right|-z\right],
\end{equation}
where ${}_{p}\Psi_{q}$ denotes the Fox-Wright function defined by \cite[p. 56, Eq. (1.11.14)]{Kilbas-Srivastava-Trujillo-Book-2006}
\begin{equation}\label{Fox-WrightF-Def}
	{}_{p}\Psi_{q}\left[\left.\begin{matrix}
		(a_j,\alpha_j)_{1,p}\\
		(b_j,\beta_j)_{1,q}
	\end{matrix}\right|z\right]:=\sum_{k=0}^{\infty}\frac{\displaystyle \prod_{j=1}^{p}\Gamma(a_j+\alpha_j k)}{\displaystyle\prod_{j=1}^{q}\Gamma(b_j+\beta_j k)}\frac{z^k}{k!}.
\end{equation}
If 
\begin{equation}\label{Delta1}
	\Delta':=\sum_{j=1}^{q}\beta_j-\sum_{\ell=1}^{p}\alpha_\ell>-1,
\end{equation}
then the series in \eqref{Fox-WrightF-Def} is absolutely convergent for all $z\in\mathbb{C}$ (see \cite[p. 56]{Kilbas-Srivastava-Trujillo-Book-2006}).

\section{Main results}\label{Main results}

Let $L^1(A)$ be the space of all Lebesgue measurable, complex-valued functions $\varphi:A\rightarrow \mathbb{C}$ with finite norm
\[
\|\varphi\|_1:=\int_{A}|\varphi(t)|\mathrm{d}t.
\]

\begin{theorem}\label{MainTh-1}	
	Let the conditions in Definition \ref{New Definition} be satisfied and let 
	\begin{equation}\label{MainTh-1-1}
		\Re(\mathfrak{c}_2(\lambda))
		>-\min[0,\Re(\mathfrak{p}_m)].
	\end{equation}
	Also, let $x^{\lambda-\delta} \varphi(x)\in L^1(\mathbb{R}_{>0})$. Then, for $\Re(s)>0$, we have
	\begin{equation}\label{MainTh-1-2}
		\mathcal{L}[x^\lambda \mathcal{I}\varphi](s)=\int_{0}^{\infty}K_{\mathcal{I}}(s,x)\varphi(x)\mathrm{d}x,
	\end{equation}
	where
	\begin{align}\label{MainTh-1-3}
		K_{\mathcal{I}}(s,x)
		\equiv \mathcal{J}x^\lambda\mathrm{e}^{-sx}
		&=x^{\lambda-\delta}\sum_{k=0}^{m}\frac{A_k}{A_0}(a)_k (b)_k\notag\\
		&\hspace{0.5cm}\cdot H_{2,3}^{3,0}\left[sx\left|\begin{matrix}
			(\mathfrak{c}_2(\lambda)+\mu-a,1/\nu),(\mathfrak{c}_2(\lambda)+\mu-b,1/\nu)\\
			(0,1),(\mathfrak{c}_2(\lambda),1/\nu),(\mathfrak{c}_2(\lambda)+\mathfrak{p}_k,1/\nu)
		\end{matrix}\right.\right].
	\end{align}
\end{theorem}
\begin{proof}
	By Fubini's theorem, it is easy to see that
	\begin{align}
		\mathcal{L}[x^\lambda \mathcal{I}\varphi](x)
		&=\int_{0}^{\infty}x^{\lambda}\mathrm{e}^{-sx}\Bigg\{\frac{{\nu}x^{-\delta-\nu(\mu+ h)}}{\Gamma(\mu)}
		\int_{0}^{x}(x^\nu-t^\nu)^{\mu-1}\notag\\
		&\hspace{1cm}\cdot{}_{r+2}F_{r+1}\left[\begin{matrix}
			a,b,\\
			\mu,
		\end{matrix}~
		\begin{matrix}
			(f_r+m_r)\\
			(f_r)
		\end{matrix};1-\frac{t^\nu}{x^\nu}\right]
		\varphi(t)t^{\nu h+\nu-1}\mathrm{d}t\Bigg\}\mathrm{d}x\notag\\
		&=\int_{0}^{\infty}\varphi(t)\Bigg\{\frac{\nu t^{\nu h+\nu-1} }{\Gamma(\mu)}\int_{t}^{\infty}(x^\nu-t^\nu)^{\mu-1}\notag\\
		&\hspace{1cm}\cdot{}_{r+2}F_{r+1}\left[\begin{matrix}
			a,b,\\
			\mu,
		\end{matrix}~
		\begin{matrix}
			(f_r+m_r)\\
			(f_r) 
		\end{matrix};1-\frac{t^\nu}{x^\nu}\right]
		\mathrm{e}^{-sx}x^{\lambda-\delta-\nu(\mu+h)}\mathrm{d}x\Bigg\}\mathrm{d}t\notag\\
		&=\int_{0}^{\infty}\varphi(t)(\mathcal{J}x^\lambda\mathrm{e}^{-sx})(t)\mathrm{d}x
		=\int_{0}^{\infty}\varphi(x)K_{\mathcal{I}}(s,x)\mathrm{d}x,
		\label{MainTh-1-Proof-0}
	\end{align}
	provided that 
\begin{align*}
	I&:=
	\int_{0}^{\infty}|\varphi(x)|t^{\Re(\lambda-\delta)}
	\Bigg\{\int_{1}^{\infty}(u^\nu-1)^{\Re(\mu)-1}
	u^{\Re(\lambda-\delta)-\nu\Re(\mu)-\nu h}
	\notag\\
	&\hspace{1cm}\cdot\left|{}_{r+2}F_{r+1}\left[\begin{matrix}
		a,b,\\
		\mu,
	\end{matrix}~
	\begin{matrix}
		(f_r+m_r)\\
		(f_r) 
	\end{matrix};1-\frac{1}{u^\nu}\right]
	\right|\mathrm{e}^{-\Re(s)ut}\mathrm{d}u\Bigg\}\mathrm{d}t<\infty.
\end{align*}
Note that 
\begin{align*}
	I&\leq 
	\int_{0}^{\infty}|\varphi(x)|t^{\Re(\lambda-\delta)}
	\mathrm{e}^{-\Re(s)t}
	\Bigg\{\int_{1}^{\infty}(u^\nu-1)^{\Re(\mu)-1}
	u^{\Re(\lambda-\delta)-\nu\Re(\mu)-\nu h}
	\notag\\
	&\hspace{1cm}\cdot\left|{}_{r+2}F_{r+1}\left[\begin{matrix}
		a,b,\\
		\mu,
	\end{matrix}~
	\begin{matrix}
		(f_r+m_r)\\
		(f_r) 
	\end{matrix};1-\frac{1}{u^\nu}\right]
	\right|\mathrm{d}u\Bigg\}\mathrm{d}t.
\end{align*}
Since $x^{\lambda-\delta} \varphi(x)\in L^1(\mathbb{R}_{>0})$ and $\Re(s)>0$, it is sufficient to guarantee the convergence of the integral
\[
I':=\int_{1}^{\infty}(u^\nu-1)^{\Re(\mu)-1}
u^{\Re(\lambda-\delta)-\nu\Re(\mu)-\nu h}
\left|{}_{r+2}F_{r+1}\left[\begin{matrix}
	a,b,\\
	\mu,
\end{matrix}~
\begin{matrix}
	(f_r+m_r)\\
	(f_r) 
\end{matrix};1-\frac{1}{u^\nu}\right]
\right|\mathrm{d}u.
\]
Recall that
\begin{equation}\label{MainTh-1-Proof-1}
	{}_{p+1}F_{p}\left[\begin{matrix}
		a_1,\cdots,a_{p+1}\\
		b_1,\cdots,b_p
	\end{matrix};1-z\right]
	=\begin{cases}
		\mathcal{O}(1), & \text{$\Re(\psi_p)>0$ or $\Re(\psi_p)=0$ $(\psi_p\neq0)$}; \\
		\mathcal{O}(z^{\Re(\psi_p)}), & \Re(\psi_p)<0;\\
		\mathcal{O}(\log z), & \psi_p=0,
	\end{cases}
\end{equation}
as $z\rightarrow 0^{+}$, where $\psi_p:=\sum_{\ell=1}^{p}b_\ell-\sum_{\ell=1}^{p+1}a_\ell$ (see \cite{Buhring-1992}). So we have
\[
{}_{r+2}F_{r+1}\left[\begin{matrix}
	a,b,\\
	\mu,
\end{matrix}~
\begin{matrix}
	(f_r+m_r)\\
	(f_r) 
\end{matrix};1-\frac{1}{u^\nu}\right]
=\begin{cases}
	\mathcal{O}(1), &  \text{$\Re(\mathfrak{p}_m)>0$ or $\Re(\mathfrak{p}_m)=0$ $(\mathfrak{p}_m\neq0)$}; \\
	\mathcal{O}(u^{-\nu\Re(\mathfrak{p}_m)}), & \Re(\mathfrak{p}_m)<0;\\
	\mathcal{O}(\log u), & \mathfrak{p}_m=0,
\end{cases}
\]
as $u\rightarrow +\infty$, where $\mathfrak{p}_m$ is defined by \eqref{pk}. In addition, 
\[
{}_{r+2}F_{r+1}\left[\begin{matrix}
	a,b,\\
	\mu,
\end{matrix}~
\begin{matrix}
	(f_r+m_r)\\
	(f_r) 
\end{matrix};1-\frac{1}{u^\nu}\right]=\mathcal{O}(1),
\]
as $u\rightarrow 1^{+}$. 
The condition \eqref{MainTh-1-1} is therefore obtained by ensuring the convergence of the integral
\[
\int_{T}^{\infty}(u^\nu-1)^{\Re(\mu)-1}
u^{\Re(\lambda-\delta)-\nu\Re(\mu)-\nu h-\nu\min[0,\Re(\mathfrak{p}_m)]}
\mathrm{d}u.
\]

We now evaluate $K_{\mathcal{I}}(s,x)
\equiv\mathcal{J}x^\lambda \mathrm{e}^{-sx}$ given by \eqref{MainTh-1-3}. Let us express $\mathrm{e}^{-sx}$ as its Mellin-Barnes integral
\begin{equation}\label{MainTh-1-Proof-2}
	\mathrm{e}^{-sx}=\frac{1}{2\pi\mathrm{i}}
	\int_{\mathrm{i}c\infty}\Gamma(z)(sx)^{-z}\mathrm{d}z,
\end{equation}
where $c>0$ and $|\arg(s)|<\pi/2$. Under the condition \eqref{MainTh-1-1}, we use \eqref{Lemma FI of Power Function II} to obtain
\begin{align}
	\mathcal{J}x^\lambda \mathrm{e}^{-sx}
	&=\frac{1}{2\pi\mathrm{i}}
	\int_{\mathrm{i}c\infty}\Gamma(z)s^{-z}\mathcal{J}x^{\lambda-z}\mathrm{d}z\notag\\
	&=x^{\lambda-\delta}\sum_{k=0}^{m}\frac{A_k}{A_0}(a)_k(b)_k
	\frac{1}{2\pi\mathrm{i}}\int_{\mathrm{i}c\infty}\frac{\Gamma(z)
		\Gamma(\mathfrak{c}_2(\lambda-z))
		\Gamma(\mathfrak{c}_2(\lambda-z)+\mathfrak{p}_k)}{
		\Gamma(\mathfrak{c}_2(\lambda-z)+\mu-a)
		\Gamma(\mathfrak{c}_2(\lambda-z)+\mu-b)}(sx)^{-z}\mathrm{d}z\notag\\
	&=x^{\lambda-\delta}\sum_{k=0}^{m}\frac{A_k}{A_0}(a)_k(b)_k
	\frac{1}{2\pi\mathrm{i}}\int_{\mathrm{i}c\infty}\frac{\Gamma(z)
		\Gamma(\mathfrak{c}_2(\lambda)+z/\nu)
		\Gamma(\mathfrak{c}_2(\lambda)+\mathfrak{p}_k+z/\nu)}{
		\Gamma(\mathfrak{c}_2(\lambda)+\mu-a+z/\nu)
		\Gamma(\mathfrak{c}_2(\lambda)+\mu-b+z/\nu)}(sx)^{-z}\mathrm{d}z.
	\label{MainTh-1-Proof-3}
\end{align}
From the definition \eqref{H-Function-Def} of the $H$-function, we have
\begin{align}
	&\frac{1}{2\pi\mathrm{i}}\int_{\mathrm{i}c\infty}\frac{\Gamma(z)
		\Gamma(\mathfrak{c}_2(\lambda)+z/\nu)
		\Gamma(\mathfrak{c}_2(\lambda)+\mathfrak{p}_k+z/\nu)}{
		\Gamma(\mathfrak{c}_2(\lambda)+\mu-a+z/\nu)
		\Gamma(\mathfrak{c}_2(\lambda)+\mu-b+z/\nu)}(sx)^{-z}\mathrm{d}z\notag\\
	&\hspace{1cm}=H_{2,3}^{3,0}\left[sx\left|\begin{matrix}
		(\mathfrak{c}_2(\lambda)+\mu-a,1/\nu),(\mathfrak{c}_2(\lambda)+\mu-b,1/\nu)\\
		(0,1),(\mathfrak{c}_2(\lambda),1/\nu),(\mathfrak{c}_2(\lambda)+\mathfrak{p}_k,1/\nu)
	\end{matrix}\right.\right].
	\label{MainTh-1-Proof-4}
\end{align}
The corresponding indexes \eqref{H-function-Condition-1}--\eqref{H-function-Condition-5} concerning the above $H$-function $H_{2,3}^{3,0}$ satisfies
\begin{equation}\label{MainTh-1-Proof-5}
	\Delta=a^{*}=a_1^{*}=\delta^{*}=1
	~~~\text{and}~~~
	\mu^{*}=-\mu-k-\frac{1}{2}~(0\leq k\leq m).
\end{equation}
Using \eqref{MainTh-1-Proof-3} and \eqref{MainTh-1-Proof-4}, we obtain \eqref{MainTh-1-3} and the result \eqref{MainTh-1-2} of Theorem \ref{MainTh-1} follows from \eqref{MainTh-1-3} and \eqref{MainTh-1-Proof-0}. This completes the proof.
\end{proof}

\begin{remark}\label{MainTh-1-Remark}
	Since the kernel $K_{\mathcal{I}}(s,x)$ is a finite sum of $H$-functions, a direct analysis of its behaviour near zero and infinity would therefore be interesting to expedite. Under the conditions given in \eqref{MainTh-1-Proof-5}, we can use Corollary 1.10.1 of \cite{Kilbas-Saigo-Book-2004} to obtain
	\[
	H_{2,3}^{3,0}\left[sx\left|\begin{matrix}
		(\mathfrak{c}_2(\lambda)+\mu-a,1/\nu),(\mathfrak{c}_2(\lambda)+\mu-b,1/\nu)\\
		(0,1),(\mathfrak{c}_2(\lambda),1/\nu),(\mathfrak{c}_2(\lambda)+\mathfrak{p}_k,1/\nu)
	\end{matrix}\right.\right]
	=\mathcal{O}\big(x^{-\Re(\mu)-k}\mathrm{e}^{x|s|\cos(\pi+\arg(s))}\big)~(x\rightarrow+\infty),
	\]
	and thus from \eqref{MainTh-1-3} we have
	\begin{equation}\label{MainTh-1-Remark-1}
		K_{\mathcal{I}}(s,x)=\big(x^{\Re(\lambda-\delta-\mu)}\mathrm{e}^{x|s|\cos(\pi+\arg(s))}\big)
		~(x\rightarrow+\infty).
	\end{equation}
	On the other hand, by using \cite[p. 61, Eq. (1.12.23)]{Kilbas-Srivastava-Trujillo-Book-2006}, we obtain
	\[
	H_{2,3}^{3,0}\left[sx\left|\begin{matrix}
		(\mathfrak{c}_2(\lambda)+\mu-a,1/\nu),(\mathfrak{c}_2(\lambda)+\mu-b,1/\nu)\\
		(0,1),(\mathfrak{c}_2(\lambda),1/\nu),(\mathfrak{c}_2(\lambda)+\mathfrak{p}_k,1/\nu)
	\end{matrix}\right.\right]
	=\mathcal{O}\big(x^{\rho_k^{*}}\big)
	~(x\rightarrow 0^{+})
	\]
	where $\rho_k^{*}:=\min[0,\nu\Re(\mathfrak{c}_2(\lambda)),\nu\Re(\mathfrak{c}_2(\lambda))+\nu\Re(\mathfrak{p}_m)]$. 
	Hence, it follows from \eqref{MainTh-1-3} that
	\begin{align}
		K_{\mathcal{I}}(s,x)
		&=\mathcal{O}\big(x^{\Re(\lambda-\delta)+\min[\rho_0^{*},\cdots,\rho_m^{*}]}\big)\notag\\
		&=\mathcal{O}\big(x^{\Re(\lambda-\delta)+\nu\min[0,\Re(\mathfrak{c}_2(\lambda)),\Re(\mathfrak{c}_2(\lambda))+\Re(\mathfrak{p}_m)]}\big)~(x\rightarrow 0^{+}).\label{MainTh-1-Remark-2}
	\end{align}
In view of the condition \eqref{MainTh-1-1}, we observe that the expression $\nu\min[0,\Re(\mathfrak{c}_2(\lambda)),\Re(\mathfrak{c}_2(\lambda))+\Re(\mathfrak{p}_m)]$ in \eqref{MainTh-1-Remark-2} equals to zero, and consequently, we infer that
\begin{equation}\label{MainTh-1-Remark-3}
	K_{\mathcal{I}}(s,x)=\mathcal{O}(x^{\Re(\lambda-\delta)})~(x\rightarrow 0^{+}).
\end{equation}
Thus, the assertions \eqref{MainTh-1-Remark-1} and \eqref{MainTh-1-Remark-3} suggest the imposition of the condition $x^{\lambda-\delta}\varphi(x)\in L^1(\mathbb{R}_{>0})$ as stated in the hypothesis of Theorem \ref{MainTh-1}. 
\end{remark}

\begin{theorem}\label{MainTh-2}
	Let the conditions in Definition \ref{New Definition} be satisfied, and let 
	\begin{equation}\label{MainTh-2-1}
		\Re(\mathfrak{c}_1(\lambda))
		>-\min[0,\Re(\mathfrak{p}_m)].
	\end{equation}
	Also, let $x^{\lambda-\delta} \varphi(x)\in L^1(\mathbb{R}_{>0})$. Then, for $\Re(s)>0$, we have 
	\begin{equation}\label{MainTh-2-2}
		\mathcal{L}[x^\lambda \mathcal{J}\varphi](s)=\int_{0}^{\infty}K_{\mathcal{J}}(s,x)\varphi(x)\mathrm{d}x,
	\end{equation}
	where
	\begin{align}
		K_{\mathcal{J}}(s,x)\equiv \mathcal{I}x^\lambda \mathrm{e}^{-sx}
		&:=x^{\lambda-\delta}\sum_{k=0}^{m}\frac{A_k}{A_0}(a)_k (b)_k\notag\\
		&\hspace{0.5cm}\cdot {}_{2}\Psi_{2}\left[\left.\begin{matrix}
			(\mathfrak{c}_1(\lambda),1/\nu), (\mathfrak{c}_1(\lambda)+\mathfrak{p}_k,1/\nu)\\
			(\mathfrak{c}_1(\lambda)+\mu-a,1/\nu),
			(\mathfrak{c}_1(\lambda)+\mu-b,1/\nu)
		\end{matrix}\right|-sx\right].\label{MainTh-2-3}
	\end{align}
\end{theorem}
\begin{proof}
	By Fubini's theorem, we have
	\begin{align}
		\mathcal{L}[x^\lambda \mathcal{J}f](x)
		&=\int_{0}^{\infty}x^{\lambda}\mathrm{e}^{-sx}\Bigg\{\frac{\nu x^{\nu h+\nu-1}}{\Gamma(\mu)}
		\int_{x}^{\infty}(t^\nu-x^\nu)^{\mu-1}\notag\\
		&\hspace{1cm}\cdot {}_{r+2}F_{r+1}\left[\begin{matrix}
			a,b,\\
			\mu,
		\end{matrix}~
		\begin{matrix}
			(f_r+m_r)\\
			(f_r)
		\end{matrix};1-\frac{x^\nu}{t^\nu}\right]
		\varphi(t)t^{-\delta-\nu(\mu+ h)}\mathrm{d}t\Bigg\}\mathrm{d}x\notag\\
		&=\int_{0}^{\infty}\varphi(t)\Bigg\{\frac{\nu t^{-\delta-\nu(\mu+h)}}{\Gamma(\mu)}\int_{0}^{t}(t^\nu-x^\nu)^{\mu-1}\notag\\
		&\hspace{1cm}\cdot{}_{r+2}F_{r+1}\left[\begin{matrix}
			a,b,\\
			\mu,
		\end{matrix}~
		\begin{matrix}
			(f_r+m_r)\\
			(f_r)
		\end{matrix};1-\frac{x^\nu}{t^\nu}\right]\mathrm{e}^{-sx}
		x^{\lambda+\nu h+\nu-1}\mathrm{d}x\Bigg\}\mathrm{d}t\notag\\
		&=\int_{0}^{\infty}\varphi(t)(\mathcal{I}x^\lambda \mathrm{e}^{-sx})(t)\mathrm{d}t
		=\int_{0}^{\infty}\varphi(x)K_{\mathcal{J}}(s,x)\mathrm{d}x,
		\label{MainTh-2-Proof-0}
	\end{align}
	provided that
\begin{align*}
	J&:=\int_{0}^{\infty}|\varphi(t)|t^{\Re(\lambda-\delta)}
	\Bigg\{\int_{0}^{1}(1-u^\nu)^{\Re(\mu)-1}u^{\Re(\lambda)+\nu h+\nu-1}\\
	&\hspace{1cm}\cdot\left|{}_{r+2}F_{r+1}\left[\begin{matrix}
		a,b,\\
		\mu,
	\end{matrix}~
	\begin{matrix}
		(f_r+m_r)\\
		(f_r)
	\end{matrix};1-u^\nu\right]\right|\mathrm{e}^{-\Re(s)tu}\mathrm{d}u\Bigg\}\mathrm{d}t<\infty.
\end{align*}
Note that
\begin{align*}
	J&\leq \int_{0}^{\infty}|\varphi(t)|t^{\Re(\lambda-\delta)}
	\Bigg\{\int_{0}^{1}(1-u^\nu)^{\Re(\mu)-1}u^{\Re(\lambda)+\nu h+\nu-1}\\
	&\hspace{1cm}\cdot\left|{}_{r+2}F_{r+1}\left[\begin{matrix}
		a,b,\\
		\mu,
	\end{matrix}~
	\begin{matrix}
		(f_r+m_r)\\
		(f_r)
	\end{matrix};1-u^\nu\right]\right|\mathrm{d}u\Bigg\}\mathrm{d}t.
\end{align*}
Since $x^{\lambda-\delta} \varphi(x)\in L^1(\mathbb{R}_{>0})$ and $\Re(s)>0$, it is sufficient to guarantee the convergence of the integral 
\[
J':=\int_{0}^{1}(1-u^\nu)^{\Re(\mu)-1}u^{\Re(\lambda)+\nu h+\nu-1}
\left|{}_{r+2}F_{r+1}\left[\begin{matrix}
	a,b,\\
	\mu,
\end{matrix}~
\begin{matrix}
	(f_r+m_r)\\
	(f_r)
\end{matrix};1-u^\nu\right]\right|\mathrm{d}u.
\]
In view of \eqref{MainTh-1-Proof-1}, we know that $J'$ is finite if the condition \eqref{MainTh-2-1} is satisfied.

Next, we evaluate $K_{\mathcal{J}}(s,x)\equiv\mathcal{I}x^\lambda \mathrm{e}^{-sx}$ involved in \eqref{MainTh-2-Proof-0}. Using the integral representation \eqref{MainTh-1-Proof-2}, we obtain
\begin{align}
	\mathcal{I}x^\lambda \mathrm{e}^{-sx}
	&=\frac{1}{2\pi\mathrm{i}}
	\int_{\mathrm{i}c\infty}\Gamma(z)s^{-z}\mathcal{I}x^{\lambda-z}\mathrm{d}z\notag\\
	&=x^{\lambda-\delta}\sum_{k=0}^{m}\frac{A_k}{A_0}(a)_k (b)_k
	\frac{1}{2\pi\mathrm{i}}\int_{\mathrm{i}c\infty}
	\frac{\Gamma(z)\Gamma(\mathfrak{c}_1(\lambda-z))
		\Gamma(\mathfrak{c}_1(\lambda-z)+\mathfrak{p}_k)}{
		\Gamma(\mathfrak{c}_1(\lambda-z)+\mu-a)
		\Gamma(\mathfrak{c}_1(\lambda-z)+\mu-b)}(xs)^{-z}\mathrm{d}z\notag\\
	&=x^{\lambda-\delta}\sum_{k=0}^{m}\frac{A_k}{A_0}(a)_k (b)_k
	\frac{1}{2\pi\mathrm{i}}\int_{\mathrm{i}c\infty}
	\frac{\Gamma(z)\Gamma(\mathfrak{c}_1(\lambda)-z/\nu)
		\Gamma(\mathfrak{c}_1(\lambda)+\mathfrak{p}_k-z/\nu)}{
		\Gamma(\mathfrak{c}_1(\lambda)+\mu-a-z/\nu)
		\Gamma(\mathfrak{c}_1(\lambda)+\mu-b-z/\nu)}(xs)^{-z}\mathrm{d}z.
	\label{MainTh-2-Proof-1}
\end{align}
In view of \eqref{H-Function-Def} and the reduction formula \eqref{H-function-to-Fox-WrightF}, we find that
\begin{align}
	&\frac{1}{2\pi\mathrm{i}}\int_{\mathrm{i}c\infty}
	\frac{\Gamma(z)\Gamma(\mathfrak{c}_1(\lambda)-z/\nu)
		\Gamma(\mathfrak{c}_1(\lambda)+\mathfrak{p}_k-z/\nu)}{
		\Gamma(\mathfrak{c}_1(\lambda)+\mu-a-z/\nu)
		\Gamma(\mathfrak{c}_1(\lambda)+\mu-b-z/\nu)}(xs)^{-z}\mathrm{d}z\notag\\
	&\hspace{1cm}=H_{2,3}^{1,2}\left[sx\left|\begin{matrix}
		(1-\mathfrak{c}_1(\lambda),1/\nu), (1-\mathfrak{c}_1(\lambda)-\mathfrak{p}_k,1/\nu)\\
		(0,1), (1-\mathfrak{c}_1(\lambda)-\mu+a,1/\nu), (1-\mathfrak{c}_1(\lambda)-\mu+b,1/\nu)
	\end{matrix}\right.\right]\notag\\
	&\hspace{1cm}={}_{2}\Psi_{2}\left[\left.\begin{matrix}
		(\mathfrak{c}_1(\lambda),1/\nu), (\mathfrak{c}_1(\lambda)+\mathfrak{p}_k,1/\nu)\\
		(\mathfrak{c}_1(\lambda)+\mu-a,1/\nu),
		(\mathfrak{c}_1(\lambda)+\mu-b,1/\nu)
	\end{matrix}\right|-sx\right].
	\label{MainTh-2-Proof-2}
\end{align}
The relation \eqref{Delta1} concerning the ${}_{2}\Psi_{2}$-function satisfies $\Delta'=0$ and the indexes \eqref{H-function-Condition-1} and \eqref{H-function-Condition-2} satisfies 
\[
\Delta=1~~~\text{and}~~~a^{*}=1.
\]
Now combining \eqref{MainTh-2-Proof-1} and \eqref{MainTh-2-Proof-2}, we get immediately \eqref{MainTh-2-3} and the desired result \eqref{MainTh-2-2} of Theorem \ref{MainTh-2} follows from \eqref{MainTh-2-3} and \eqref{MainTh-2-Proof-0}. This proof completes the proof. 
\end{proof}

\begin{remark}
As in Remark \ref{MainTh-1-Remark}, we give a direct analysis of the behaviour of the kernel $K_{\mathcal{J}}(s,x)$ involved in the integral operator \eqref{MainTh-2-2}. Using the formula \cite[p. 11, Eq. (1.5.13)]{Kilbas-Saigo-Book-2004}, we have
\[
H_{2,3}^{1,2}\left[sx\left|\begin{matrix}
	(1-\mathfrak{c}_1(\lambda),1/\nu), (1-\mathfrak{c}_1(\lambda)-\mathfrak{p}_k,1/\nu)\\
	(0,1), (1-\mathfrak{c}_1(\lambda)-\mu+a,1/\nu), (1-\mathfrak{c}_1(\lambda)-\mu+b,1/\nu)
\end{matrix}\right.\right]
=\mathcal{O}\big(x^{\rho_k}\big)~(x\rightarrow+\infty),
\]
where $\rho_k:=-\nu\min[\Re(\mathfrak{c}_1(\lambda)),\Re(\mathfrak{c}_1(\lambda))+\Re(\mathfrak{p}_m)]$ and $|\arg(s)|<\pi/2$. Thus, 
\begin{align}\label{MainTh-2-Remark-1}
	K_{\mathcal{J}}(s,x)
	&=\mathcal{O}\big(x^{\Re(\lambda-\delta)+\max[\rho_0,\cdots,\rho_m]}\big) \notag\\
	&=\mathcal{O}\big(x^{\Re(\lambda-\delta)-\nu\Re(\mathfrak{c}_1(\lambda))-\nu\min[0,\Re(\mathfrak{p}_m)]}\big) \notag\\
	&=\mathcal{O}\big(x^{\Re(\lambda-\delta)}\big)~(x\rightarrow+\infty).
\end{align}
On the other hand, we have
\[
H_{2,3}^{1,2}\left[sx\left|\begin{matrix}
	(1-\mathfrak{c}_1(\lambda),1/\nu), (1-\mathfrak{c}_1(\lambda)-\mathfrak{p}_k,1/\nu)\\
	(0,1), (1-\mathfrak{c}_1(\lambda)-\mu+a,1/\nu), (1-\mathfrak{c}_1(\lambda)-\mu+b,1/\nu)
\end{matrix}\right.\right]
=\mathcal{O}(1)~(x\rightarrow0^{+}),
\]
and therefore
\begin{equation}\label{MainTh-2-Remark-2} 
	K_{\mathcal{J}}(s,x)=\mathcal{O}\big(x^{\Re(\lambda-\delta)}\big)
	~(x\rightarrow 0^{+}).
\end{equation}
The relations \eqref{MainTh-2-Remark-1} and \eqref{MainTh-2-Remark-2} therefore suggest the condition that $x^{\lambda-\delta}\varphi(x)\in L^1(\mathbb{R}_{>0})$, as given in the hypothesis of Theorem \ref{MainTh-2}. It is worth mentioning here that the same asymptotic behaviour can also be obtained by using the theory of the Fox-Wright function; see Paris and Kaminski \cite[p. 57, Case (i)]{Paris-Kaminski-Book-2001}. 
\end{remark}

Finally, we conclude this section by pointing out that Theorems \ref{MainTh-1} and \ref{MainTh-2} provide generalizations to the results of Srivastava \emph{et al.} \cite[p. 6, Theorem 3]{Srivastava-Saigo-Raina-1993}.

\end{document}